\documentclass[letterpaper,reqno,11pt]{amsart}

\usepackage{preamble}
\usepackage{mathtools}

\usepackage{lipsum}
\usepackage{subcaption}

\title[Root polytopes, generalized tropical oriented matroids and matching ensembles]{Subdivisions of root polytopes, generalized tropical oriented matroids, and generalized matching ensembles}

\author{Tianyu Jiang, Yuan Yao, Qingyue Yu, and Chenyi Zhang}

\begin{document}

\begin{abstract}
    We study a generalization of tropical oriented matroids by Ardila and Develin, and show that they are in bijection with subdivisions of root polytopes, which are sub-polytopes of a product of two simplices. We also study a generalization of a matching-based combinatorial model by Oh and Yoo in the special case of triangulations.
\end{abstract}

\maketitle

\section{Introduction}

Tropical oriented matroids are first defined by Ardila and Develin as a combinatorial model for tropical hyperplane arrangements, as a tropical analogue of the relationship between ordinary oriented matroids and hyperplane arrangements (\cite{ardila2007tropical}). It has since been shown (see \cref{thm:bij_tom}) that these matroids are in bijection with subdivisions of a product of two simplices, an object of significant independent interest in combinatorics and algebraic geometry (\cite{GKZ1994discriminants}, \cite{postnikov2009permutohedra}).

In this paper, we study a generalization of tropical oriented matroids that are first proposed by Oh and Yoo (\cite{oh2011triangulations}), and show that they are similarly in bijection with subdivisions of root polytopes, defined by taking the convex hull of a subset of vertices of a product of two simplices. Via the Cayley trick, these subdivisions are also in bijection with mixed subdivisions of certain generalized permutohedra (\cite{postnikov2009permutohedra}).

In the case of triangulations (of product of two simplicies), there are several more additional combinatorial models, including trianguloids by Galashin et al.\ (\cite{galashin2018trianguloids}) and matching ensembles by Oh and Yoo (\cite{oh2013ensembles}). In particular, trianguloids are already generalized to model triangulations of root polytopes. Following the footsteps of previous work by the second author (\cite{yao2025linkage}), we explain how to generalize matching ensembles for root polytopes.

This paper is organized as follows. In \cref{sec:subdiv} and \cref{sec:gtom}, we recall some basic definitions regarding subdivisions of root polytopes and generalized permutohedra, (generalized) tropical oriented matroids, and mention some known connections between these objects. From \cref{sec:subdiv_to_gtom} to \cref{sec:disconn_gtom}, we prove the correspondence between generalized tropical oriented matroids and subdivisions of root polytopes, as well as a few related properties of generalized tropical oriented matroids. In \cref{sec:gen_ensem} we briefly introduce and generalize matching ensembles.

\section{Subdivisions of root polytopes}\label{sec:subdiv}

Let $n$ and $d$ be positive integers. Throughout this paper, we will frequently use the two sets of indices $[n] = \{1, 2, \dots, n\}$ and $[\bar{d}] = \{\bar{1}, \bar{2}, \dots, \bar{d}\}$, distinguished by a bar.

Let $\Delta^{d-1}$ denote the standard $(d-1)$-dimensional unit simplex in $\mathbb{R}^d$, whose vertices are $e_{\bar{1}}, e_{\bar{2}}, \dots, e_{\bar{d}}$ (here, $e_{\bar{j}} = (0, 0, \dots, 1, \dots, 0)$, with a $1$ in the $\bar{j}$-th coordinate and $0$ in other coordinates). We will also use $\Delta_{\bar{J}}$ (where $\bar{J}\subseteq [\bar{d}]$) to denote a face of this simplex of the form $\text{conv}(\{e_{\bar{j}}\mid \bar{j}\in \bar{J}\})$.

\begin{definition}[\cite{santos2003cayley}, Section 1.1]
    Let $Q$ be a polytope. A \emph{cell} of $Q$ is a sub-polytope (of the same dimension as $Q$) whose vertices are a subset of those of $Q$. A \emph{polyhedral subdivision} of $Q$ is a collection of cells which cover $Q$ and pairwise intersect properly,  meaning that $C\cap C'$ is a common face of both $C$ and $C'$ for every pair of cells $C$ and $C'$ in the subdivision. It is a \emph{triangulation} if all cells are simplices.
\end{definition}

\begin{definition}[\cite{santos2003cayley}, Section 1.2]
    Given a Minkowski sum of polytopes $P = \sum_{i=1}^n P_i$, a \emph{Minkowski cell} is a full-dimensional polytope $C$ of the form $\sum_{i=1}^n C_i$, where the vertices of $C_i$ is a nonempty subset of that of $P_i$. It is \emph{fine} if it cannot be subdivided into smaller Minkowski cells.
     
    A polyhedral subdivision of $P$ into (fine) Minkowski cells is called a \emph{(fine) mixed subdivision} of $P$ if for any two cells $C = \sum_{i=1}^n C_i$ and $C' = \sum_{i=1}^n C'_i$, we have $C \cap C' = \sum_{i=1}^n (C_i \cap C'_i)$, and $C_i \cap C'_i$ is a common face of $C_i$ and $C'_i$.
\end{definition}

In this paper we are mainly interested in subdivisions of the following polytopes.

\begin{definition}[\cite{postnikov2009permutohedra}, Sections 9 and 12]
    Let $G\subseteq K_{n,d}$ be a bipartite graph whose left vertices are indexed by $[n]$ and the right vertices are indexed by $[\bar{d}]$. Define a \emph{root polytope} $Q_G$ to be the convex hull of $\{e_i + e_{\bar{j}}\mid (i, \bar{j})\in E(G)\}$ in $\R^{n+d}$.\footnote{The original definition uses $e_i - e_{\bar{j}}$ to better resemble the type-A roots, but the resulting polytope is congruent.} Define a \emph{generalized permutohedron} $P_G$ to be the Minkowski sum $\sum_{i=1}^n \Delta_{\bar{J}_i}$, where $\bar{J}_i\subseteq [\bar{d}]$ is the set of neighbors of vertex $i$ in $G$.
\end{definition}

\begin{example}
    When $G$ is the complete bipartite graph $K_{n,d}$, $Q_G$ can also be viewed as a Cartesian product of two simplices $\Delta^{n-1}\times \Delta^{d-1}$, and $P_G$ is a dilated simplex $n\Delta^{d-1}$. In general, $Q_G$ and $P_G$ are contained in $\Delta^{n-1}\times \Delta^{d-1}$ and $n\Delta^{d-1}$, respectively.
\end{example}

\begin{figure}[h!]
    \centering
    \includegraphics[height=1.5in]{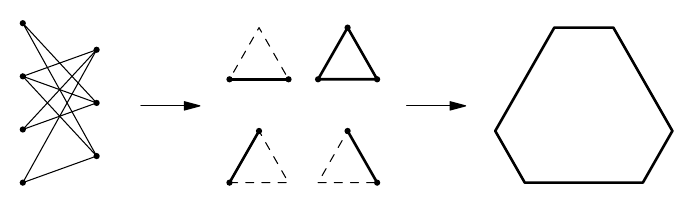}
    \caption{A graph $G$ (left) and the corresponding $P_G$ (right), formed by taking the Minkowski sums of $4$ faces of $\Delta^2$ (middle). Here $n = 4$ and $d = 3$.}\label{fig:pg}
\end{figure}

The well-known \emph{Cayley trick} connects the subdivisions of these two polytopes.

\begin{theorem}[\cite{postnikov2009permutohedra}, Corollary 14.6]
    Polyhedral subdivisions of $Q_G$ are in bijection with mixed subdivisions of $P_G$, and triangulations correspond to fine mixed subdivisions. In particular, the mixed subdivision can be obtained by intersecting the subdivision of $Q_G$ with the affine subspace $\{\frac{1}{n}\mathbf{1}_{[n]}\}\times \R^d$ and dilate the result by a factor of $n$.
\end{theorem}

\begin{figure}[h!]
    \centering
    \includegraphics[height=1.5in]{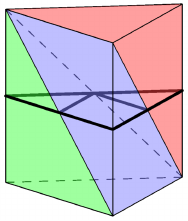}
    \hspace{1in}
    \includegraphics[height=1.4in]{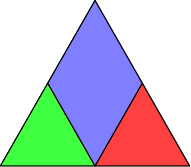}

    \caption{A subdivision of $\Delta^{1}\times \Delta^{2}$ and the corresponding subdivision of $2\Delta^{2}$.}
\end{figure}

Note that the cells of $Q_G$ and Minkowski cells of $P_G$ also take the form of $Q_{G'}$ and $P_{G'}$ respectively, where $G'$ is a spanning/connected subgraph of $G$. Therefore, subdivisions of $Q_G$ can be encoded by a collection of graphs.

\begin{proposition}[\cite{ardila2007tropical}, Lemma 6.1 and \cite{galashin2018trianguloids}, Proposition 7.5]\label{prop:facet}
    Any facet of $Q_G$ is of the form $Q_H$, where $H$ is a maximal disconnected subgraph of $G$ with two connected components $I_1\sqcup \bar{J}_1$ and $I_2\sqcup \bar{J}_2$ (where $I_1\sqcup I_2 = [n]$ and $\bar{J}_1\sqcup \bar{J}_2 = [\bar{d}]$), and $G\setminus H$ only consist of edges between $I_1$ and $\bar{J}_2$, or only edges between $I_2$ and $\bar{J}_1$. If $G\setminus H$ contains both types of edges then $Q_H$ is contained in $Q_G$ but not on the boundary.
\end{proposition}

\begin{definition}
    Two subgraphs $G$ and $G'$ of $K_{n,d}$ are \emph{compatible} if whenever there is a cycle $C$ in $G\cup G'$ which alternates between edges of $G$ and edges of $G'$, both $G$ and $G'$ contain $C$.\footnote{It is permissible for this cycle to repeat vertices, as it can be decomposed into two alternating cycles.}
\end{definition}

\begin{theorem}[\cite{ardila2007tropical}, Theorem 6.2, generalized]\label{thm:subdiv_graphs}
A collection of connected subgraphs $G_1, \dots, G_k$ of $G$ encodes a subdivision of $Q_G$ if and only if:
\begin{itemize} 
\item (Facet-sharing) If $H$ is a (maximal) disconnected subgraph of $G_s$ such that $Q_{H}$ is a facet of $Q_{G_s}$ but not contained in a facet of $Q_G$, then there exists another graph $G_t$ that contains $H$.
\item (Compatibility) The subgraphs are pairwise compatible.

If all subgraphs are cyclic, then $\mathcal{G}$ encodes a triangulation of $Q_G$.
\end{itemize}
\end{theorem}

\begin{remark}\label{rem:alt_connect}
    It is not difficult to see from these conditions that if two compatible graphs $Q_{G_s}$ and $Q_{G_t}$ share a facet $Q_H$, then either $G_s$ connects $I_1$ and $\bar{J}_2$ and $G_t$ connects $I_2$ and $\bar{J}_1$ (using notation of \cref{prop:facet}), or vice versa.
\end{remark}

\section{(Generalized) tropical oriented matroids}\label{sec:gtom}

For the purpose of this paper, we consider a \emph{tropical hyperplane} to be a translated copy of the normal fan of $\Delta^{d-1}$, which lie in the \emph{tropical projective space} $\mathbb{TP}^{d-1} = \R^d / \R \mathbf{1}_{[\bar{d}]}$. Notice that a tropical hyperplane divides the space into $d$ (open) full-dimensional \emph{sectors} and $2^d-1$ \emph{faces}, which are naturally indexed by non-empty subsets of $[\bar{d}]$. 

Given $n$ tropical hyperplanes, the position of any point with respect to 

\begin{definition}[\cite{ardila2007tropical}, Definition 3.1, 3.4]
    An \emph{$(n,d)$-type} is an $n$-tuple $A = (A_1, \dots, A_n)$ of nonempty subsets of $[\bar{d}]$. Each $i\in [n]$ will be referred to as a \emph{position}. If each $A_i$ is a singleton, then $A$ is called a \emph{tope}.
    
    The \emph{refinement} of a type $A$ with respect to an ordered partition $P = (P_1, \dots, P_r)$ of $[\bar{d}]$ is a type $A_P$, obtained by replacing each coordinate $A_i$ with its intersection with the earliest part in the partition for which the intersection is nonempty. If $r = d$ (i.e.\ all parts are singletons, which equivalently form a permutation of $[\bar{d}]$) then this refinement is \emph{total}, and if $r = 2$ then this refinement is \emph{coarse}.
\end{definition}

\begin{remark}
    It is clear that an $(n,d)$-type corresponds to a subgraph $G(A)$ of $K_{n,d}$ (where none of the vertices in $[n]$ are isolated). Therefore, we will treat types as the same as bipartite graphs from now on.
\end{remark}

The axioms of tropical oriented matroids resemble the (co)vector axioms of ordinary oriented matroids.

\begin{definition}[\cite{ardila2007tropical}, Definition 3.5]
    An \emph{$(n,d)$-tropical oriented matroid (TOM)} is a collection $\mathcal{O}$ of $(n,d)$-types satisfying the following axioms: \begin{itemize}
        \item (Boundary) For each $\bar{j}\in [\bar{d}]$, the type $(\{\bar{j}\}, \dots, \{\bar{j}\})$ is in $\mathcal{O}$.
        \item (Surrounding) The collection is closed under refinements.
        \item (Comparability) The types (treated as graphs) are pairwise compatible.\footnote{The original definition uses certain comparability graphs and requires them to be acyclic; it is not difficult to see that the definition here is equivalent.}
        \item (Elimination) For any two types $A, B \in \mathcal{O}$ and a position $i\in [n]$, there exists a type $C\in \mathcal{O}$ such that $C_i = A_i\cup B_i$ and $C_{i'} \in \{A_{i'}, B_{i'}, A_{i'}\cup B_{i'}\}$ for all $i'\in [n]$. We say that $C$ is obtained by \emph{eliminating between $A$ and $B$ at position $i$}.
    \end{itemize}
\end{definition}

Similarly, deletion and contraction operations can be defined on tropical oriented matroids.

\begin{definition}[\cite{ardila2007tropical}, Proposition 4.7 and 4.8]\label{def:minors}
    Let $\mathcal{O}$ be an $(n,d)$-tropical oriented matroid.    
    For $i\in [n]$, the \emph{$i$-deletion} of $\mathcal{O}$, denoted $\mathcal{O}_{\setminus i}$, is an $(n-1, d)$-tropical oriented matroid formed by deleting the $i$-th coordinate from each type of $\mathcal{O}$.
    For $\bar{j}\in [\bar{d}]$, the \emph{$\bar{j}$-contraction} of $\mathcal{O}$, denoted $\mathcal{O}_{/\bar{j}}$, is an $(n,d-1)$-tropical oriented matroid formed by taking all types of $\mathcal{O}$ that does not contain $\bar{j}$ in any coordinate.

    For subsets $I\subseteq [n]$ and $\bar{J}\subseteq [\bar{d}]$, the \emph{$(I, \bar{J})$-minor} of $\mathcal{O}$, denoted $\mathcal{O}|_{I, \bar{J}}$, is obtained by deleting and contracting all elements in $[n]$ and $[\bar{d}]$ that are not in $I$ or $\bar{J}$, respectively. (It is clear that the deletions and contractions can be done in any order.)
\end{definition}

While not all tropical oriented matroids correspond to tropical hyperplane arrangements, they can all be represented by tropical \emph{pseudo}-hyperplane arrangements, via a tropical analogue of the Topological Representation Theorem (see \cite{Horn2016troporep}). Such arrangements can also be obtained by taking the Poincar\'{e} dual of a mixed subdivision of $n\Delta^{d-1}$.

\begin{figure}[h!]
    \centering
    \includegraphics[height=2.5in]{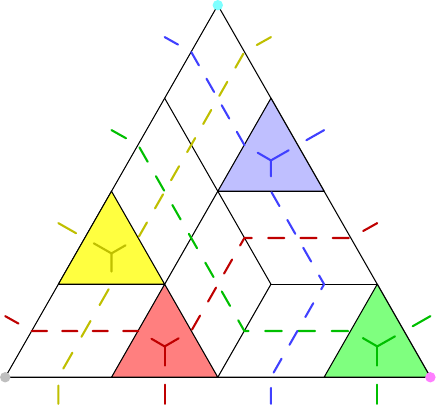}
    \caption{An arrangement of $4$ tropical pseudo-hyperplanes in $\mathbb{TP}^2$ in dashed lines, obtained from a mixed subdivision of $4\Delta^2$.}
\end{figure}

In fact, many of the axioms of tropical oriented matroids have direct geometrical interpretations in the context of mixed subdivisions: \begin{itemize}
    \item The boundary types correspond to the vertices of $n\Delta^{d-1}$, and topes correspond to all vertices of the subdivision;
    \item A subgraph $H$ is a refinement of $G$ if and only if $P_H$ is a face of $P_G$ (\cite{ardila2007tropical}, Proposition 6.4), where the vertices of $P_G$ correspond to total refinements and facets correspond to certain coarse refinements;
    \item Compatibility of types (or compatibility of connected subgraphs) corresponds to proper intersection of Minkowski cells.
\end{itemize} 

The correspondence between Elimination axiom and the Facet-sharing condition is less clear, but has been established as well. This gives a full bijection:

\begin{theorem}[\cite{ardila2007tropical}, Theorem 6.3, \cite{oh2011triangulations}, Corollary 4.13 and \cite{Horn2016troporep}, Corollary 7.12]\label{thm:bij_tom}
    There is a bijection between $(n,d)$-tropical oriented matroids and mixed subdivisions of $n\Delta^{d-1}$ (and subdivisions of $\Delta^{n-1}\times \Delta^{d-1}$), where types correspond to the faces of the mixed subdivision.
\end{theorem}

Motivated by this bijection, a generalization of tropical oriented matroid has been proposed with some modifications to the original axioms: 

\begin{definition}[\cite{oh2011triangulations}, Definition 5.1]
    Given a subgraph $G$ of $K_{n,d}$, a \emph{$G$-tropical oriented matroid (GTOM)} is a collection $\mathcal{O}$ of $(n,d)$-types satisfying the Surrounding, Comparability, and Elimination axioms of a tropical oriented matroid, as well as the following axioms: \begin{itemize}
        \item (Subgraph) Each type is a subgraph of $G$.
        \item (Generalized Boundary) Every total refinement of $G$ is a type in $\mathcal{O}$.
    \end{itemize}
\end{definition}

It is not difficult to see that when $G = K_{n,d}$ we recover the original definition of a tropical oriented matroid (since $(\{\bar{j}\}, \dots, \{\bar{j}\})$ can be obtained by any total refinement where the first part is $\{\bar{j}\}$). In general, total refinements of $G$ correspond to vertices of $P_G$.

In the language of tropical pseudo-hyperplane arrangements (arising as Poincar\'{e} duals of mixed subdivisions), this generalization represents moving the centers (or apexes) of some pseudo-hyperplanes to infinity in some directions, hence dividing the finite space into fewer than $d$ sectors (sometimes referred to as \emph{degeneration}).

\begin{figure}[h!]
    \centering
    \includegraphics[height=2.5in]{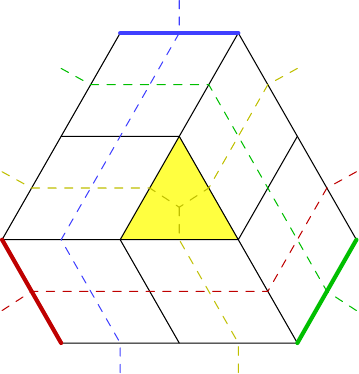}

    \caption{A fine mixed subdivision of $P_G$ (from \cref{fig:pg}) and the corresponding tropical pseudo-hyperplane arrangement. Pseudo-hyperplanes that do not correspond to full-dimensional simplices are not centered inside $P_G$.}
\end{figure}

In the following sections, we prove the following:

\begin{theorem}\label{thm:main}
    There is a bijection between $G$-tropical oriented matroids and mixed subdivision of $P_G$ (and subdivisions of $Q_G$), where types correspond to the faces of the mixed subdivision.
\end{theorem}

We will first prove the case where $G$ is connected (corresponding to $P_G$ being $(d-1)$-dimensional), and then show that the disconnected case can be reduced to a ``direct sum'' of connected components.

\section{From subdivisions to GTOM}\label{sec:subdiv_to_gtom}

One direction of the bijection is easy using existing tools of polytope subdivisions.

\begin{proposition}\label{prop:sub_to_gtom}
    Given a mixed subdivision of $P_G$, taking all types that correspond to faces of the mixed subdivision gives a $G$-tropical oriented matroid.
\end{proposition}

\begin{proof}
    Using the Cayley trick, we construct a subdivision of $Q_G$ that corresponds to the given mixed subdivision, and then extend it to a subdivision of $\Delta^{n-1}\times \Delta^{d-1}$ by placing the missing vertices $v\in \Delta^{n-1}\times \Delta^{d-1}$ one at a time, coning over the ``visible'' part of the current subdivision from $v$ to get a larger subdivision. (See \cite{toth2017handbook}, Section 16.2 for precise definitions.) This gives us a mixed subdivision of $n\Delta^{d-1}$ that extends the mixed subdivision of $P_G$, and (by \cref{thm:bij_tom}) an ordinary tropical oriented matroid $\mathcal{O}'$ that contains the set $\mathcal{O}$ of all the types for the mixed subdivision of $P_G$.

    Note that $\mathcal{O}$ consists of exactly the types of $\mathcal{O}'$ that are subgraphs of $P_G$, so the Subgraph, Surrounding, and Comparability axioms are automatically satisfied for $\mathcal{O}$, inherited from $\mathcal{O}'$. For Elimination, note that if $A$ and $B$ are subgraphs of $G$ then so is $A\cup B$, so the type $C$ furnished by Elimination axiom of $\mathcal{O}'$ is also in $\mathcal{O}$, being a subgraph of $A\cup B$. Generalized Boundary axiom is also satisfied, as any total refinement of $G$ corresponds to a vertex of $P_G$, and there is a unique way to express each vertex of $P_G$ as a Minkowski sum $\sum_{i=1}^n e_{\bar{j}_i}$ where $(i, \bar{j}_i)\in E(G)$, so it must appear as a type in $\mathcal{O}$. This shows that $\mathcal{O}$ forms a $G$-tropical oriented matroid.
\end{proof}

\section{Generating types in a GTOM}\label{sec:type_in_gtom}

As mentioned in \cref{sec:gtom}, the connected types of a GTOM $\mathcal{O}$ are clearly pairwise compatible, so to show the reverse direction of \cref{thm:main} it suffices to show that they satisfy the facet-sharing condition as well. To this end, we need to develop some technical tools to show that certain types are in a GTOM using elimination starting from boundary types (and compatibility with some given type).

We first state an easy lemma, which is similar to Lemma 4.5 of \cite{ardila2007tropical}:

\begin{lemma}\label{lem:agree}
    Let $A$ and $B$ be two compatible types, and $\bar{J}\subseteq [\bar{d}]$ be a set of vertices in the same connected component of $A\cap B$. For all $i\in [n]$ for which $A_i\cap \bar{J}$ and $B_i\cap \bar{J}$ are both non-empty, these two sets are the same.
\end{lemma}

\begin{proof}
    Suppose that $\bar{j}_1 \in A_i\cap \bar{J}$ and $\bar{j}_2\in B_i\cap \bar{J}$ for $\bar{j}_1\neq \bar{j}_2$, then it is clear that there is a cycle alternating between edges of $A$ and edges of $B$, using $(i, \bar{j}_1)\in E(A)$, $(i, \bar{j}_2)\in E(B)$, as well as some edges in $A\cap B$ to go from $\bar{j}_2$ back to $\bar{j}_1$ (which may revisit vertex $i$). For $A$ and $B$ to be compatible they must both contain this cycle, so these two edges must be in both graphs, as desired.
\end{proof}

We will say that $A$ and $B$ \emph{agree on $\bar{J}\subseteq [\bar{d}]$ at position $i\in [n]$} if either $A_i\cap \bar{J} = B_i\cap \bar{J}$ or at least one of the two sets are empty.

\begin{proposition}\label{prop:generate}
    Any connected type $A$ of a $G$-tropical oriented matroid $\mathcal{O}$ can be generated using a series of eliminations between types, starting from only the boundary types of $G$.
\end{proposition}

The statement of this proposition may seem trivial from definition, but the explicit construction in the proof will be reused later to generate certain connected components of other types.

\begin{proof}
    For this proof we implicitly use induction on $n$ and $d$; that is, we assume that the proposition already holds for ``smaller'' GTOMs.

    By reordering the vertices in $[\bar{d}]$ if necessary, we can assume that for all $\bar{j}\in [\bar{d}]$, the vertices of $\bar{1}, \dots, \bar{j}$ are in the same connected component if the vertices $\overline{j+1}, \dots, \bar{d}$ are removed from (the graph of) $A$.

    We now construct a series of types $A^{(\bar{j})}$ for each $\bar{j} = \bar{1}, \bar{2}, \dots, \bar{d}$ such that: \begin{itemize}
        \item If $G_i$ does not contain any element in $[\bar{j}]$, then $A^{(\bar{j})}_i$ is a singleton set consisting of the smallest element of $G_i$.
        \item If $A_i$ does not contain any element in $[\bar{j}]$ but $G_i$ does, then $A^{(\bar{j})}_i$ is a subset of $[\bar{j}]$.
        \item If $A_i$ contains an element in $[\bar{j}]$, then $A^{(\bar{j})}_i = A_i\cap [\bar{j}]$.
    \end{itemize}

    It is clear that letting $A^{(\bar{1})}$ be the boundary type of $\mathcal{O}$ that is obtained with the total refinement $(\{\bar{1}\}, \{\bar{2}\}, \dots, \{\bar{d}\})$ satisfy the conditions above. For each $\bar{j}\geq \bar{2}$ we perform the following algorithm:

    {\color{gray} [Throughout this algorithm, we will also use the following example, where $\bar{j} = \bar{9}$:
    
    \[\begin{array}{cccccccccc}
        i & 1 & 2 & 3 & 4 & 5 & 6 & 7 & 8 & 9 \\
        G= & (\bar{2}\bar{9}, & \bar{1}\bar{2}\bar{4}\bar{9}, & \bar{1}\bar{2}\bar{5}, & \bar{1}\bar{2}\bar{3}\bar{5}\bar{6}, & \bar{3}\bar{4}\bar{7}, & \bar{2}\bar{3}\bar{8}, & \bar{6}\bar{7}, & \bar{2}\bar{4}\bar{7}\bar{8}, & \bar{3}\bar{8}\bar{9}),\\
        A= & (\bar{2}\bar{9}, & \bar{1}\bar{2}, & \bar{1}\bar{5}, & \bar{2}\bar{3}\bar{6}, & \bar{3}\bar{4}, & \bar{2}\bar{3}, & \bar{6}\bar{7}, & \bar{7}\bar{8}, & \bar{8}).\\
    \end{array}\]

    For this example $A^{(\bar{8})}$ is just $A$ with $\bar{9}$ missing from the first position.
    
    Every instance of this example being referred to will be marked by square brackets.]}

    \begin{figure}[h!]
        \centering
        \includegraphics[height=1.5in]{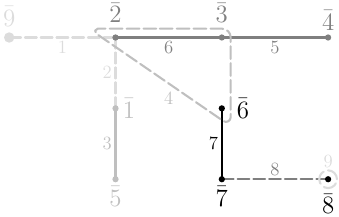}
        \caption{An example connected type $A$ to be generated via elimination, represented as a hypergraph on $[\bar{j}] = [\bar{9}]$. The positions in $I = [9]$ (represented by (hyper)edges) and indices in $[\bar{j}]$ are color-coded based on their labeled level, with darker colors corresponding to higher levels. Dashed edges correspond to opposing positions and solid edges correspond to agreeing poitions.}
    \end{figure}

    \begin{itemize}
        \item \textbf{Step 0}: Call the positions in which $A$ has at least one element in $[\bar{j}]$ \emph{active} and the rest \emph{inactive}. Let $I\subseteq [n]$ be the set of active positions. Note that $A$ restricted to $I\sqcup [\bar{j}]$ is connected by our assumption. {\color{gray} [All positions are active in the example.]}
        
        \item \textbf{Step 1}: By our assumption, there exists an active position where $\bar{j}$ appears together with an element of $[\overline{j-1}]$; we say that this position is the \emph{root position}.
        
        We label the active positions as ``opposing'' or ``agreeing'' based on $G$ and $A$: \begin{itemize}
            \item \textbf{Step 1.0}: Let $\bar{J}_0 = \{\bar{j}\}$. Label all positions $i\in I$ for which $G_i\cap \bar{J}_0\neq \varnothing$ as \emph{level-0 opposing}. Denote this set by $I_0^{\neq}$. (Note that this includes the root position.)

            {\color{gray} [In the example, the level-0 opposing positions are $\{1, 2, 9\}$.]}
            
            \item \textbf{Step 1.1}: Remove all vertices in $I_0^{\neq}$ and $\bar{J}_0$ from $A$. This may disconnect $A$ into several components. Take one of the connected components (with at least one vertex in $I\setminus I_0^{\neq}$) which is connected to $\bar{J}_0$ in $A$ via only vertices in $I_0^{\neq}$ and vertices adjacent to $I_0^{\neq}$. Such component must exist by the connectivity of $A$ in our assumption. We denote this component as $I_1^{=}\sqcup \bar{J}_1$. The positions in $I_1^{=}$ are labeled \emph{level-1 agreeing}. Label all unlabeled positions $i$ for which $G_i\cap \bar{J}_1\neq \varnothing$ as \emph{level-1 opposing}. Denote this set by $I_1^{\neq}$.

            {\color{gray} [In the example, after removing positions $\{1, 2, 9\}$ from $A$, there are two connected components, and we pick the one with vertices $3, \bar{1}, \bar{5}$, where $\bar{1}$ is connected to $\bar{9}$ via $2, \bar{2}, 1$, so $I_1^= = \{3\}$, $\bar{J}_1 = \{\bar{1}, \bar{5}\}$, and $I_1^{\neq} = \{4\}$.]}
            
            \item \textbf{Step 1.2+}: Remove vertices in $I_1^{=}, I_1^{\neq}, \bar{J}_1$ from the remainder of $A$, and define $I_2^{=}$, $I_2^{\neq}$, $\bar{J}_2$ similar to the previous step. (When choosing a connected component, we may connect it to any vertex in $\bar{J}_0$ or $\bar{J}_1$ through only vertices in $I_0^{\neq}, I_1^{=}, I_1^{\neq}$ and adjacent vertices.) Repeat until all positions in $I$ are labeled.

            {\color{gray} [In the example, we can pick $I_2^{=} = \{5, 6\}, \bar{J}_2 = \{\bar{2}, \bar{3}, \bar{4}\}$, which makes position $8$ level-2 opposing. Finally position $7$ is level-3 agreeing with $\bar{J}_3 = \{\bar{6}, \bar{7}\}$.]}
        \end{itemize}

        Note that $I_0^{\neq}, I_1^{=}, I_1^{\neq}, \dots$ form a partition of $[n]$, while $\bar{J}_0, \bar{J}_1, \dots$ are disjoint but do not necessarily form a partition of $[\bar{j}]$. Let $\bar{J}_{\infty}$ be the elements of $[\bar{j}]$ that are not covered by this process. {\color{gray} [In the example, $\bar{J}_{\infty} = \{\bar{8}\}$.]}

        \item \textbf{Step 2}: We claim that there exists a type $B^{(\bar{j})} \in \mathcal{O}$ with the following properties: \begin{itemize}
            \item For $k\geq 1$, $B^{(\bar{j})}_i = A^{(\overline{j-1})}_i \subseteq \bar{J}_k$ for each level-$k$ agreeing position $i$;
            \item For $k\geq 0$, $B^{(\bar{j})}_i$ is a nonempty subset of $\bar{J}_k$ for each level-$k$ opposing position $i$. Note these positions in $B^{(\bar{j})}$ are disjoint from the corresponding positions in $A^{(\overline{j-1})}$ unless both of them are equal to $\{\bar{j}\}$ (this happens if and only if $k = 0$ and $\bar{j}$ is the smallest element of $A$ in that position).
        \end{itemize}

        {\color{gray} [In the example, the type will have the form 
        \[\begin{array}{cccccccccc}
        i & 1 & 2 & 3 & 4 & 5 & 6 & 7 & 8 & 9 \\
        B^{(\bar{9})} = & (\bar{9}, & \bar{9}, & \mathbf{\bar{1}\bar{5}}, & \bar{1}?\bar{5}?, & \mathbf{\bar{3}\bar{4}}, &\mathbf{\bar{2}\bar{3}}, & \mathbf{\bar{6}\bar{7}}, & \bar{2}?\bar{4}?, & \bar{9}),\\
        \end{array}\]
        where the agreeing positions have been bolded.]}

        In particular, this type can be generated the set of boundary types given by a permutation starting with $\bar{j}$, followed by a permutation of $\bar{J}_1$, then a permutation of $\bar{J}_2$, \dots, a permutation of $\bar{J}_{\infty}$, and finally the elements of $[\bar{d}]\setminus [\bar{j}]$ in order. By our labeling algorithm, a level-$k$ position of such a boundary type consists of a single element from $\bar{J}_k$, and $A$ restricted to the level-$k$ agreeing positions and elements of $\bar{J}_k$ is a connected type in a smaller GTOM $\mathcal{O}_k$. The graphs corresponds to the $\mathcal{O}_k$ are not connected to each other, so we can generate $B^{(\bar{j})}$ by eliminating between the boundary types of $\mathcal{O}_k$ separately and then combining the results. (If generating each smaller connected type involves $b_k$ boundary types in $\mathcal{O}_k$, then generating $B^{(\bar{j})}$ involves $\prod_{k} b_k$ boundary types in $\mathcal{O}$.)

        {\color{gray} [In the example, the lexicographically smallest boundary type used to generate $B^{(\bar{9})}$ is \[(\bar{9}, \bar{9}, \bar{1}, \bar{1}, \bar{3}, \bar{2}, \bar{6}, \bar{2}, \bar{9}),\] by refining $G$ with permutation $(\bar{9}\ |\ \bar{1}, \bar{5}\ |\ \bar{2}, \bar{3}, \bar{4}\ |\ \bar{6}, \bar{7}\ |\ \bar{8})$. The three smaller GTOMs involve indices $\{3, \bar{1}, \bar{5}\}$, $\{5, 6, \bar{2}, \bar{3}, \bar{4}\}$, and $\{7, \bar{6}, \bar{7}\}$ respectively. Position $4$ is not fully fixed, but must be a subset of $\{\bar{1}, \bar{5}\}$ since only these two elements appear in that position among all relevant boundary types.]}

        \item \textbf{Step 3}: Now eliminate between $A^{(\overline{j-1})}$ and $B^{(\bar{j})}$ at the root position to get a new type $A'^{(\overline{j})}$. In active positions, it is clear that $A'^{(\overline{j})}$ is a subset of $[\bar{j}]$. In inactive positions, notice that all boundary types involved in generating $A^{(\overline{j-1})}$ and $B^{(\bar{j})}$ so far are refinements of $G$ that are refinements of the partition $([\bar{j}], \{\overline{j+1}\}, \dots, \{\bar{d}\})$, so these positions in $A'^{(\overline{j})}$ are all either subsets of $[\bar{j}]$ or a singleton not in $[\bar{j}]$.
        
        We now show that $A'^{(\overline{j})}$ agrees with $A$ on $[\bar{j}]$ on all active positions. Since $A$ restricted to positions in $I$ is connected, we may reorder the positions in $[n]$ such that $I = [\ell]$ and $1$ is the root position, and $A$ restricted to the first $\ell'(\leq \ell)$ positions is still connected. By our choice of connected components in the labeling algorithm, we may additionally require that all level-$k$ opposing positions and all level-$(k+1)$ agreeing positions come after level-$k$ agreeing positions in this ordering (for $k\geq 1$).

        {\color{gray} [The positions in the example are already ordered to have the above properties.]}

        We now process the active positions in numerical order: \begin{itemize}
            \item By definition of elimination, we have $A'^{(\overline{j})}_1 = A^{(\overline{j-1})}_1 \cup B^{(\bar{j})}_1 = A^{(\overline{j-1})}_1 \cup \{\bar{j}\}$, which agrees with $A_1\cap [\bar{j}]$.
            \item For an agreeing position $i > 1$, note that $A^{(\overline{j-1})}_i$ = $B^{(\bar{j})}_i$, so $A'^{(\overline{j})}_i = A^{(\overline{j-1})}_i$, which agrees with $A$ since $A_i$ does not contain $\bar{j}$ (or else $i$ would be level-0 opposing).
            \item For a (level-$k$) opposing position $i > 1$, if $A^{(\overline{j-1})}_i = B^{(\bar{j})}_i = \{\bar{j}\}$ then we have that $\{\bar{j}\}$ is the smallest element of $A_i$ so $A'^{(\overline{j})}$ agrees with $A$. Otherwise, $A^{(\overline{j-1})}_i(=A_i\cap[\overline{j-1}])$ and $B_i^{(\bar{j})} (\subseteq \bar{J}_k)$ are disjoint, and there exists an element $\bar{a}\in A^{(\overline{j-1})}_i$ that is already connected to all vertices in $\bar{J}_k$ (and hence an element $\bar{b}\in B_i^{(\bar{j})}$) through vertices that are less than $i$ in $A$ (and $A'^{(\overline{j})}$, by induction on $i$), due to our ordering of the positions. 
            
            Since $A$ and $A'^{(\bar{j})}$ are compatible, by \cref{lem:agree} they must agree on $\{\bar{a}, \bar{b}\}$ at position $i$. Since $A'^{(\bar{j})}_i \in \{A^{(\overline{j-1})}_i, B_i^{(\bar{j})}, A^{(\overline{j-1})}_i\cup B_i^{(\bar{j})}\}$, it must contain at least one of $\bar{a}$ and $\bar{b}$, and hence must contain $\bar{a}$ since $\bar{a}\in A_i$, so it must contain the entire $A^{(\overline{j-1})}_i = A_i\cap[\overline{j-1}]$. Whether $A'^{(\bar{j})}_i$ also contains $\bar{b}$ depends on whether $\bar{b} = \bar{j}$ and whether $A_i$ contains $\bar{j}$ (both can only happen if $k = 0$), but it is not difficult to see that in any case $A'^{(\bar{j})}$ agrees with $A$.

            {\color{gray} [In the example, consider position $i = 8$, which is a level-2 opposing position. By elimination $A'^{(\bar{9})}_8$ must be $A^{(\bar{5})}_8 = \{\bar{7},\bar{8}\}$, $B^{(\bar{9})}_8 \subseteq \{\bar{2},\bar{4}\}$, or a union of both. There is already a path in $A$ from $\bar{7}$ to both $\bar{2}$ and $\bar{4}$ using edges incident to vertices less than or equal to $8$, so $A'^{(\bar{9})}_8$ must agree with $A_8$ by the lemma.]}
        \end{itemize}

        Therefore, we see that $A'^{(\bar{j})}$ satisfy every condition listed before the algorithm, so we can let $A^{(\bar{j})} = A'^{(\bar{j})}$ and proceed to the next iteration.
    \end{itemize}

    At the end of the algorithm (when $\bar{j} = \bar{d}$), it is clear that we have $A^{(\bar{d})} = A$ and is generated from elimination starting with only boundary types, as desired.
\end{proof}

\begin{remark}
    The main generalization from the ordinary case is the construction of type $B^{(\bar{j})}$, which in an ordinary TOM will just be the boundary type $(\{\bar{j}\}, \dots, \{\bar{j}\})$. The desired properties listed in Step 2 necessitates the much more complex construction in Step 1 (in an ordinary TOM, all positions will be level-0 opposing).
\end{remark}

\section{From GTOM to Subdivisions}\label{sec:gtom_to_subdiv}

We are now ready to prove the other direction of the bijection.

\begin{proposition}\label{prop:conn}
    In a $G$-tropical oriented matroid $\mathcal{O}$, any type is contained in a connected type.
\end{proposition} 

Note that when $G = K_{n,d}$, this is Proposition 4.6 of \cite{ardila2007tropical}. The following proof is essentially a generalization following the same line of argument.

\begin{proof}
    It suffices to show that any disconnected type is strictly contained in a ``larger'' type. Take a disconnected $(n,d)$-type $A$ where one connected component is $I\sqcup \bar{J}$. We consider two cases: 

    \begin{itemize}
        \item \textbf{Case 1}: If $I = [n]$, then clearly $\bar{J}$ is a strict subset of $[\bar{d}]$. Without loss of generality we assume that $\bar{J} = [\overline{d-1}]$. Since $G$ is connected, there is a position $i$ for which $G_i$ contains both $\bar{d}$ and an element of $\bar{J}$. Using position $i$ as the root position, we can perform one iteration of the same algorithm as in the proof of the \cref{prop:generate} (with $\bar{j} = \bar{d}$, up to Step 2) to construct a type $B$ in $\mathcal{O}$ such that $B_{i'}$ is either the same as $A_{i'}$ or disjoint from $A_{i'}$ for all $i'\in [n]$. We will construct a sequence of types $B^{(0)} = B, B^{(1)}, \dots$ such that the last one strictly contains $A$ via the following algorithm for $k = 0, 1, \dots$: \begin{itemize}
            \item Call a position $i$ \emph{bad} if $A_i$ and $B^{(k)}_i$ are disjoint, and \emph{good} if $B^{(k)}_i$ contains $A_i$ (and \emph{very good} if the containment is strict). We are done if all positions are good and at least one position is very good.
            \item Eliminate between $A$ and $B^{(k)}$ at a bad position (when $k = 0$, this can be the root position) to yield type $B^{(k+1)}$. Note that the bad position must be very good after the elimination, and all good positions remain good. (Elimination will not cause any position to be neither good nor bad.)
        \end{itemize}
        Since the number of bad positions strictly decreases at each step, this algorithm terminates in finitely many steps, with the last eliminated position being very good. This yields a type that strictly contains $A$ as desired. (In fact, this type must contain $\bar{d}$ in some position, or else it cannot be compatible with $A$ due to \cref{lem:agree}.)

        \item \textbf{Case 2}: If $I \neq [n]$, then there exists a choice of connected components such that there is an edge in $G$ between $[n]\setminus I$ and $\bar{J}$. Without loss of generality we assume that $\bar{J} = [\bar{j}]$ for some $\bar{j} < \bar{d}$, and that there is only one other connected component $I'\sqcup \bar{J}'$ (where $I = [n]\setminus I, \bar{J}' = [\bar{d}]\setminus [\bar{j}]$). We can run the algorithm from the proof of \cref{prop:generate} to construct a type $A^{(\bar{j})}\in \mathcal{O}$ in that are the same as $A$ in positions of $I$. Note that this type must contain some element of $\bar{J}$ in some position of $I'$, and also that this construction works regardless of the ordering of $\bar{J}'$. Letting $\bar{J}_0 = \bar{J}$, we can again run an iteration of the algorithm (restricted to $I'$) up to Step 2 to construct a type $B$ in $\mathcal{O}$ that are the same as $A$ in positions of $I$, and are either the same or disjoint from $A$ in positions of $I'$ (with at least one position being disjoint).\footnote{The proof that the algorithm works with $\bar{J}_0$ being a set can be adapted from the original proof.} This now reduces to the previous case. (As with the previous case, the new type that strictly contains $A$ must contain an edge between $I'$ and $\bar{J}$.)
    \end{itemize}
\end{proof}

\begin{remark}
    Case 1 of the above proof is in fact the same as Case 2 with $I = \varnothing$.
\end{remark}

\begin{proof}[Proof of \cref{thm:main} for connected $G$]
    As mentioned previously, it is sufficient to verify that the connected types of a GTOM satisfy the facet-sharing condition. Take a maximal disconnected subgraph $H$ such that $Q_H$ is not contained in a facet of $Q_G$. By \cref{prop:facet}, $H$ has two connected components $I_1\sqcup \bar{J}_1$ and $I_2\sqcup \bar{J}_2$, and $G$ has edges both between $I_1$ and $\bar{J}_2$ and between $I_2$ and $\bar{J}_1$. This means that we can run Case 2 of the proof of \cref{prop:conn} with either of $(I_1, \bar{J}_1)$ and $(I_2, \bar{J}_2)$ as $(I, \bar{J})$ (and the other as $(I', \bar{J}')$) to construct two types $H_1$ and $H_2$, both containing $H$, such that $H_1$ has some edge connecting $I_1$ and $\bar{J}_2$, $H_2$ has some edge connecting $I_2$ and $\bar{J}_1$, and then extend them to two connected types. By \cref{rem:alt_connect}, these two types must be distinct, so we showed that any internal facet in the subdivision is shared by two cells, as desired.
\end{proof}

\section{GTOM for disconnected graphs}\label{sec:disconn_gtom}

We now prove that \cref{thm:main} holds for disconnected $G$, by showing that mixed subdivisions of $P_G$ and $G$-tropical oriented matroids can be separately written as ``direct sums'' of smaller mixed subdivisions and GTOMs. 

For convenience, assume that $G$ can be decomposed into two disconnected subgraphs $G_1$ and $G_2$, supported on vertices $I_1\sqcup \bar{J}_1$ and $I_2 \sqcup \bar{J}_2$ respectively.

The claim about mixed subdivisions of $P_G$ follows from a more general statement about mixed subdivisions of sums of polytopes in independent affine subspaces.

\begin{definition}\label{def:ms_sum}
    Let $P = \sum_{i\in I} P_i$ and $P' = \sum_{i'\in I'} P'_{i'}$ be two Minkowski sums of polytopes, such that $P$ and $P'$ are in independent affine subspaces of $\mathbb{R}^d$. Given mixed subdivisions $\mathcal{C}$ and $\mathcal{C}'$ of $P$ and $P'$ respectively, let the \emph{sum of the two subdivisions} $\mathcal{C} + \mathcal{C}'$ be a mixed subdivision $\hat{\mathcal{C}}$ of $\hat{P} = P + P'$ consisting of all cells in $\{C + C'\mid C\in \mathcal{C}, C'\in \mathcal{C}'\}$, with summands indexed by $I\sqcup I'$.
\end{definition}

It is not difficult to check that $\mathcal{C} + \mathcal{C}'$ satisfies the conditions of a mixed subdivision.

\begin{proposition}\label{prop:ms_decomp}
    Let $P, P', \hat{P}$ be the same as in \cref{def:ms_sum}. Any mixed subdivision of $\hat{P}$ can be (uniquely) written as a sum of a mixed subdivision of $P$ and a mixed subdivision of $P'$. Similarly, any fine mixed subdivision of $\hat{P}$ is a sum of fine mixed subdivisions.
\end{proposition}

\begin{proof}
    Let $v'$ be any vertex of $P'$. A mixed subdivision $\hat{\mathcal{C}}$ of $\hat{P} = P + P'$ induces a mixed subdivision of $P + \{v'\}$ (as a face of $\hat{P}$), so we can obtain a mixed subdivision $\mathcal{C}$ of $P$, and similarly a mixed subdivision $\mathcal{C}'$ for $P'$. 
    
    For any cell $\hat{C} = C + C'$ in $\hat{\mathcal{C}}$ (where $C$ corresponds to the summands indexed by $[n]$), if $C$ is not a cell of $\mathcal{C}$, then there is a cell $C_0$ in $\mathcal{C}$ that it does not intersect properly with (as Minkowski cells of $P$). We can find a cell $\hat{C}_0 \in \hat{\mathcal{C}}$ that contains the face $C_0 + v'$, and hence contains $C_0$ as part of the Minkowski sum; this cell clearly does not intersect properly with $\hat{C}$ as Minkowski cells, which contradicts the fact that $\hat{\mathcal{C}}$ is a mixed subdivision. Therefore, all cells in $\hat{\mathcal{C}}$ belongs to $\mathcal{C} + \mathcal{C}'$, and by a standard volume argument the reverse also holds.

    Showing that the sum is unique and the specialization to fine mixed subdivisions is easy.
\end{proof}

Now, notice that $P_G = P_{G_1} + P_{G_2}$, where $P_{G_1}$ lies in the affine subspace $\sum_{\bar{j}\in \bar{J}_1} x_{\bar{j}} = |I_1|$ and $x_{\bar{j}} = 0$ for all $\bar{j}\notin \bar{J}_1$, which is clearly independent from the one in which $P_{G_2}$ lies. By \cref{prop:ms_decomp} this means that any mixed subdivision of $P_G$ is the sum of a mixed subdivision of $P_{G_1}$ and one of $P_{G_2}$. 

For $G$-tropical oriented matroids, we first prove that deletion and contraction from usual tropical oriented matroids generalize to this case (regardless of connected-ness of $G$).

\begin{proposition}
    Using the same definition as in \cref{def:minors}, $G$-tropical oriented matroids are closed under taking $(I, \bar{J})$-minors (as long as the $(I\sqcup \bar{J})$-induced subgraph of $G$ has no isolated vertices in $I$).
\end{proposition}

\begin{proof}
    Note that it is always possible to remove a given subset of vertices in $G$ in some order that never creates an isolated vertex in $[n]$ during the process, so it suffices to show that GTOMs are closed under deletion and contraction. The Subgraph axiom is trivially satisfied since these operations simply remove vertices from $G$. The preservation of Surrounding, Comparability, and Elimination axioms is proved identically as with the proof for ordinary tropical oriented matroids. Therefore the only nontrivial axiom to check is (Generalized) Boundary.

    Suppose that $T$ is a boundary type of $G_{\setminus i}$ (obtained by removing vertex $i$), obtained by a permutation (or a total refinement) $\sigma$ of $[\bar{d}]$. Since vertex $i$ is connected to at least one vertex in $[\bar{d}]$ in $G$, the same permutation $\sigma$ refines $G$ to a boundary type $\hat{T}$ which will become $T$ under the $i$-deletion. Therefore, all required boundary types appear after an $i$-deletion.

    Now suppose that $T$ is a boundary type of $G_{/\bar{j}}$ (obtained by removing vertex $\bar{j}$), obtained by a permutation $\sigma$ of $[\bar{d}]\setminus \{\bar{j}\}$. Note that for any vertex $i\in [n]$, if its only neighbor is $\bar{j}$ then $i$ is isolated in $G_{/\bar{j}}$, so we can assume that every vertex in $[n]$ has a neighbor that is not $\bar{j}$. Therefore, if we let $\hat{\sigma}$ be a permutation of $[\bar{d}]$ by appending $\bar{j}$ at the end of $\sigma$, then refining $G$ by $\hat{\sigma}$ gives $T$, which will also appear after $\bar{j}$-contraction. Therefore, all required boundary types appear after a $\bar{j}$-contraction.
\end{proof}

\begin{definition}\label{def:gtom_sum}
    Let $\mathcal{O}$ and $\mathcal{O}'$ be two generalized tropical oriented matroids, whose corresponding graphs $G$ and $G'$ are supported on two disjoint sets of vertices $I\sqcup \bar{J}$ and $I'\sqcup \bar{J}'$. The \emph{sum of two generalized tropical oriented matroids} $\mathcal{O} + \mathcal{O}'$ is a $(G\sqcup G')$-tropical oriented matroid $\hat{\mathcal{O}}$ consisting of all types in $\{(T, T')\mid T\in \mathcal{O}, T'\in \mathcal{O}'\}$.
\end{definition}

It is not difficult to check that $\mathcal{O} + \mathcal{O}'$ satisfies the axioms of a generalized tropical oriented matroid. Using a similar argument as \cref{prop:ms_decomp}, we can show the following:

\begin{proposition}\label{prop:gtom_decomp}
    Let $G$ and $G'$ be the same as in \cref{def:gtom_sum}. Any $(G\sqcup G')$-tropical oriented matroid $\hat{\mathcal{O}}$ can be (uniquely) written as a sum of a $G$-tropical oriented matroid and a $G'$-tropical oriented matroid.
\end{proposition}

Therefore, \cref{thm:main} for disconnected $G$ follows from ``taking a sum of bijections for each connected component of $G$''.

\section{Generalized matching ensembles}\label{sec:gen_ensem}

If we restrict ourselves to the case of triangulations of root polytopes, the graphs that correspond to simplices of the triangulation will all be spanning trees of $G$ (spanning forests if $G$ is disconnected). In particular, this means that the compatibility condition in \cref{thm:subdiv_graphs} can be simplified significantly:

\begin{proposition}
    Two acyclic subgraphs of $G$ are compatible if and only if whenever they both contain a perfect matching between two sets of vertices $I\subseteq [n]$ and $\bar{J}\subseteq [\bar{d}]$ (of equal size), the two matchings are the same.
\end{proposition}

(Generalized) tropical oriented matroids that correspond to triangulations are also called \emph{generic}, as they correspond to generic tropical pseudo-hyperplane arrangements.

When $G = K_{n,d}$, Oh and Yoo defined \emph{matching ensembles}, a collection of perfect matchings that encode the entire triangulation.

\begin{definition}[\cite{FieldLattice}, Definition 5.1]\label{def:ensemble}
    For positive integers $n$ and $d$, a \emph{$(n,d)$-matching stack} $\mathcal{M}$ is a collection of bijections (or \emph{partial matchings}) $M_{I, \bar{J}}$ between $I$ and $\bar{J}$ for every pair of subsets $I\subseteq [n]$ and $\bar{J}\subseteq [\bar{d}]$ such that $|I| = |\bar{J}|$. 

    A matching stack is a \emph{matching ensemble} if the following conditions hold: \begin{enumerate}
        \item (Closure) Given $I'\subseteq I$ and $\bar{J}'\subseteq \bar{J}$, if $M_{I, \bar{J}}$ contains a matching between $I'$ and $\bar{J}'$, then $M_{I', \bar{J}'} \subseteq M_{I, \bar{J}}$. (Equivalently, all matchings are pairwise compatible.)
        \item (Left linkage) Given subsets $I\subseteq [n]$ and $\bar{J}\subseteq [\bar{d}]$ such that $|I| = |\bar{J}| + 1$, the union of $M_{I', \bar{J}}$ over all $I'\subset I$ of size $|\bar{J}|$ is a tree $\mathbb{T}_{I, \bar{J}}$ whose right degree vector is equal to $\mathbf{2}_{\bar{J}}$.
        \item (Right linkage) Given subsets $I\subseteq [n]$ and $\bar{J}\subseteq [\bar{d}]$ such that $|I| + 1 = |\bar{J}|$, the union of $M_{I, \bar{J}'}$ over all $\bar{J}'\subset \bar{J}$ of size $|I|$ is a tree $\bar{\mathbb{T}}_{I, \bar{J}}$ whose left degree vector is equal to $\mathbf{2}_{I}$.
    \end{enumerate}

    The trees $\mathbb{T}_{I, \bar{J}}$ and $\bar{\mathbb{T}}_{I, \bar{J}}$ are called \emph{(left/right) linkage covectors}.
\end{definition}

The recently established connections between generic (non-generalized) tropical oriented matroids and matching ensembles in \cite{yao2025linkage} motivates defining a generalization of matching ensembles that is cryptomorphic to generic GTOMs.

To this end, we first establish a generalization of tope-linkage property from Lemma 3.10 of \cite{oh2013ensembles}, which roughly states that ``in a fine mixed subdivision of $n\Delta^{d-1}$, for any lattice point that is not $n\cdot e_{\bar{j}}$, there is an edge connecting it to a point one step further in the $e_{\bar{j}}$ direction''. In the generalized version, the point $n\cdot e_{\bar{j}}$ may not be a point in $P_G$, so we need to replace it by all points that are furthest in the $e_{\bar{j}}$ direction.

\begin{definition}
    For a $G$-tropical oriented matroid $\mathcal{O}$ and a proper nonempty subset $\bar{J}\subset [\bar{d}]$, the \emph{$\bar{J}$-boundary} of $G$ is the graph $G_{(\bar{J}, [\bar{d}]\setminus \bar{J})}$ obtained by a coarse refinement of $G$ with respect to $(\bar{J}, [\bar{d}]\setminus \bar{J})$, and the \emph{$\bar{J}$-boundary} of $\mathcal{O}$ is the set of all types in $\mathcal{O}$ that are subgraphs of $G_{(\bar{J}, [\bar{d}]\setminus \bar{J})}$.
\end{definition}

In the language of mixed subdivisions, the $\bar{J}$-boundary of a $G$-tropical oriented matroid corresponds to the face of $P_G$ supported by the vector $\mathbf{1}_{\bar{J}}$, which is clearly a mixed subdivision itself. Therefore, the $\bar{J}$-boundary of a $G$-tropical oriented matroid is a $G_{(\bar{J}, [\bar{d}]\setminus \bar{J})}$-tropical oriented matroid.

\begin{proposition}[Generalized tope-linkage]\label{prop:gen_link}
    For a generic $G$-tropical oriented matroid $\mathcal{O}$, an index $\bar{j}\in [\bar{d}]$, and a tope $T\in \mathcal{O}$ that is not on the $\{\bar{j}\}$-boundary of $\mathcal{O}$, there is another tope $T'\in \mathcal{O}$ that can be obtained by replacing one coordinate $\bar{j}_0\neq \bar{j}$ in $T$ with $\bar{j}$.
\end{proposition}

\begin{proof}
    From the previous section, we may assume that $G$ is connected (and hence $P_G$ is full-dimensional), since otherwise we can work with just the connected component containing $\bar{j}$. Let $\mathcal{C}$ be the fine mixed subdivision of $P_G$ corresponding to $\mathcal{O}$, where $T$ corresponds to a lattice point in $P_G$. 
    
    If $T$ is on some $\bar{J}$-boundary of $\mathcal{O}$ where $\bar{j}\in \bar{J}$, then it suffices to find $T'$ in the same $\bar{J}$-boundary. Observe that $G_{(\bar{J}, [\bar{d}]\setminus \bar{J})}$ is a disconnected graph since each vertex $i\in [n]$ is either connected only to vertices in $\bar{J}$ or connected only to vertices not in $\bar{J}$, so we can decompose the $\bar{J}$-boundary of $\mathcal{O}$ into a sum of GTOMs of lower dimension, one of which contains $\bar{j}$. Note that this smaller GTOM necessarily contains some other index of $\bar{J}$, or else $T$ would lie on the $\{\bar{j}\}$-boundary of $\mathcal{O}$. Using induction on dimension of $P_G$ (where the cases $d \leq 2$ are easy), we can find a desired tope in the smaller GTOM with $\bar{j}$, which translates to the desired $T'$.
    
    Otherwise, using the proof of \cref{prop:sub_to_gtom}, we can extend $\mathcal{C}$ into a fine mixed subdivision $\hat{\mathcal{C}}$ of $n\Delta^{d-1}$ and hence an ordinary $(n, d)$-tropical oriented matroid $\hat{\mathcal{O}}$ that extends $\mathcal{O}$. 
    Lemma 3.10 of \cite{oh2013ensembles} furnishes a tope $T'\in \hat{\mathcal{O}}$ that satisfies the desired properties; since $T$ is not in any $\bar{J}$-boundary of $\mathcal{O}$ where $\bar{j}\in \bar{J}$, it is not difficult to see that $T'$ will also be in $\mathcal{O}$ since its corresponding point will remain in $P_G$.
\end{proof}

\begin{figure}[h!]
    \centering
    \includegraphics[height=3.5in]{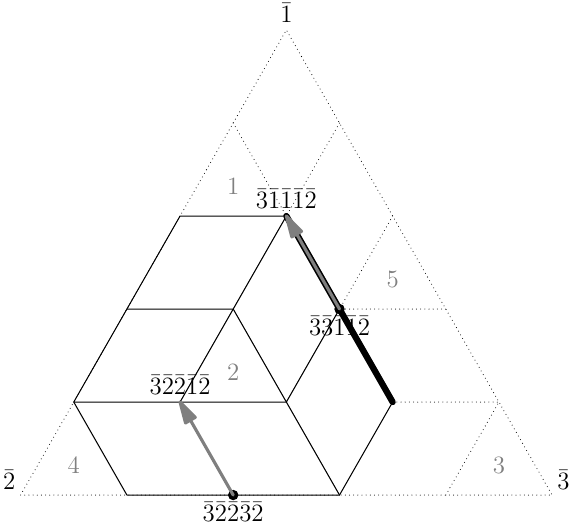}

    \caption{Two examples of the generalized tope-linkage property, where for each tope marked with a dot we find another tope that is one step further in the $e_{\bar{1}}$ direction (marked with gray arrows). The dotted lines show an extension of the fine mixed subdivision of $P_G$ (in thin solid lines) to a fine mixed subdivision of $5\Delta^2$ in which we can invoke ordinary tope-linkage. The tope $\bar{3}\bar{3}\bar{1}\bar{1}\bar{2}$ lies on the $\{\bar{1}, \bar{3}\}$-boundary of $P_G$, so it suffices to show tope-linkage along this boundary (marked with bold edge).}
\end{figure}

For generalizing matching ensembles, one must note that since not all pairs of subsets of vertices may have a matching in $G$, the collections of partial matchings are only defined for the pairs that do, and hence the left/right linkage axioms from \cref{def:ensemble} need to be modified accordingly.

\begin{definition}[\cite{galashin2018trianguloids}, Section 4]
    Given a subgraph $G$ of $K_{n,d}$, define the \emph{matching support set} $IJ_G$ to be the set of all pairs $(I, \bar{J})$ of subsets $I\subseteq [n]$ and $\bar{J}\subseteq [\bar{d}]$ such that there exists a partial matching between $I$ and $\bar{J}$ in $G$.
\end{definition}

\begin{definition}
    For a subgraph $G$ of $K_{n,d}$, a \emph{$G$-matching stack} $\mathcal{M}$ is a collection of partial matchings $M_{I, \bar{J}}$ between $I$ and $\bar{J}$ for every pair $(I, \bar{J})\in IJ_G$, such that each $M_{I, \bar{J}}$ is a subgraph of $G$.

    A $G$-matching stack is a \emph{$G$-matching ensemble} if the following conditions hold: \begin{enumerate}
        \item (Closure) Given $I'\subseteq I$ and $\bar{J}'\subseteq \bar{J}$, if $M_{I, \bar{J}}$ contains a matching between $I'$ and $\bar{J}'$, then $M_{I', \bar{J}'} \subseteq M_{I, \bar{J}}$. (Equivalently, all matchings are pairwise compatible.)
        \item (Left linkage) Given subsets $I\subseteq [n]$ and $\bar{J}\subseteq [\bar{d}]$ such that $|I| = |\bar{J}| + 1$, the union of $M_{I', \bar{J}}$ over all $I'\subset I$ for which $(I', \bar{J})\in IJ_G$ is a forest $\mathbb{T}_{I, \bar{J}}$.
        \item (Right linkage) Given subsets $I\subseteq [n]$ and $\bar{J}\subseteq [\bar{d}]$ such that $|I| + 1 = |\bar{J}|$, the union of $M_{I, \bar{J}'}$ over all $\bar{J}'\subset \bar{J}$ for which $(I, \bar{J}')\in IJ_G$ is a forest $\bar{\mathbb{T}}_{I, \bar{J}}$.
    \end{enumerate}

    The forests $\mathbb{T}_{I, \bar{J}}$ and $\bar{\mathbb{T}}_{I, \bar{J}}$ are called \emph{(left/right) linkage covectors}.
\end{definition}

\begin{remark}\label{rem:gen_linkage}
    It is not difficult to see that the left degrees of the right linkage covector $\bar{\mathbb{T}}_{I, \bar{J}}$ are all at most $2$. Therefore, unless the right linkage covector is empty, it consists of a ``normal'' right linkage tree plus some vertex-disjoint edges representing shared matches across all relevant partial matchings. Therefore, much of the combinatorics of ordinary linkage axioms directly carry over as long as we disregard these additional edges.

    \begin{figure}[h!]
    \centering
    \includegraphics[height=2in]{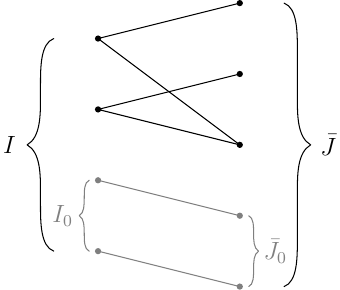}
    \caption{A right linkage covector $\bar{\mathbb{T}}_{I, \bar{J}}$ in a $G$-matching ensemble. The shared edges across matchings are in gray.}
    \end{figure}
    
    For a non-empty right linkage covector $\bar{\mathbb{T}}_{I, \bar{J}}$, suppose that $I_0$ is the set of vertices in $I$ with degree $1$ and $\bar{J}_0$ is the set of vertices matched with $I_0$ (which must also all have degree $1$), then $(I, \bar{J}')\in IJ_G$ for $\bar{J}'\subset \bar{J}$ if and only if $|I| = |\bar{J}'|$ and $\bar{J}_0\subseteq \bar{J}'$. It is also clear that there can be no edges at all between $I_0$ and $\bar{J}\setminus \bar{J}_0$ in $G$: if there is such an edge $(i, \bar{j})$, then there exists a partial matching between $I$ and $\bar{J}\setminus \{\bar{j}'\}$ using this edge and the edges in $\bar{\mathbb{T}}_{I, \bar{J}}$, where $\bar{j}'\in \bar{J}_0$ that is matched with $i$ in $\bar{\mathbb{T}}_{I, \bar{J}}$, which contradicts the fact that $(i, \bar{j})$ is not in $\bar{\mathbb{T}}_{I, \bar{J}}$.

    From this, we can use the equivalence of linkage axioms established in \cite{Fields} (also summarized in Section 2.3 of \cite{yao2025linkage}) to conclude that right linkage is equivalent to the following condition, which much more closely resemble the one used in the original definition in \cite{oh2013ensembles}: \begin{itemize}
        \item (Weak right exchange) Given $(I, \bar{J})\in IJ_G$ and $\bar{j}'\in [\bar{d}]\setminus \bar{J}$ that is connected to at least one vertex of $I$ in $G$, there exists an edge $(i, \bar{j})$ in $M_{I, \bar{J}}$ such that replacing it with $(i, \bar{j}')$ gives another matching in $\mathcal{M}$.
    \end{itemize}
    
    The symmetric versions of the statements above also hold for the left linkage covector $\mathbb{T}_{I, \bar{J}}$.
\end{remark}

The same \emph{extraction method} as described in \cite{oh2013ensembles} now allows us to obtain a $G$-matching ensemble from a generic $G$-tropical oriented matroid.

\begin{proposition}
    Given a generic $G$-tropical oriented matroid $\mathcal{O}$, for any $I\subseteq [n]$ and $\bar{J}\subseteq [\bar{d}]$ for which $(I, \bar{J})\in IJ_G$, let $M_{I, \bar{J}}$ be the unique tope of $\mathcal{O}|_{I, \bar{J}}$ that is a perfect matching. The matchings $\{M_{I, \bar{J}}\}$ form a $G$-matching ensemble $\mathcal{M}$.
\end{proposition}

\begin{proof}
    It is clear that closure holds for $\mathcal{M}$, so by matroid duality it suffices to show that $\mathcal{M}$ satisfies the weak right exchange property in the above remark. The same arguments in the proof of \cite{oh2013ensembles}, Proposition 4.3 work essentially verbatim for this purpose; the only change of note is that the added requirement that $\bar{j}'$ being connected to at least one vertex of $I$ in $G$ is equivalent to the condition that $M_{I, \bar{J}}$ is not on the $\{\bar{j}'\}$-boundary of $\mathcal{M}|_{I, [\bar{d}]}$ for us to use generalized tope-linkage (\cref{prop:gen_link}) instead of ordinary tope-linkage in the original proof.
\end{proof}

The reverse direction of showing that all $G$-matching ensembles encode generic $G$-tropical oriented matroids follows the same path as outlined in \cite{yao2025linkage} with minimal modifications (as noted in \cref{rem:gen_linkage}), giving us a full cryptomorphism:

\begin{theorem}
    Generic $G$-tropical oriented matroids are in bijection with $G$-matching ensembles.
\end{theorem}

\bigskip

\paragraph{\bfseries Acknowledgments.} 
The bulk of this project was performed during PKU Algebra and Combinatorics Experience (PACE) 2025 at Beijing International Center for Mathematical Research (BICMR). We would like to thank Yibo Gao for hosting the program and supporting the project.

\bibliographystyle{alpha}
\bibliography{refs}

\end{document}